\documentclass[12pt,twoside]{article}
\usepackage{amsmath}
\usepackage{amsthm}
\usepackage{amssymb}
\usepackage[margin=1in]{geometry}
\usepackage{graphicx}
\usepackage{ifpdf}

\ifpdf
 \DeclareGraphicsExtensions{.pdf}
\else
 \DeclareGraphicsExtensions{.eps}
\fi

\title{Increasing stability of inverse source problem for one dimensional domain }
\author{Shahah Almutairi$^a$ \qquad Ajith Gunaratne $^b$ }
\date{March 2019}

\begin{document}
\maketitle
\newtheorem{theorem}{Theorem}[section]
\newtheorem{lemma}[theorem]{Lemma}
\newtheorem{corollary}[theorem]{Corollary}
\newtheorem{definition}[theorem]{Definition}
\newtheorem{proposition}[theorem]{Proposition}

{\bf Abstract.}

Here we are investigating the one dimensional inverse source problem for Helmholtz equation where the source function is compactly supported in our domain. We show that increasing stability possible using multi-frequency wave at the two end points. Our main result is a stability estimate consists of two parts: the data discrepancy and the high frequency tail.

\vspace{10pt}


\vspace{10pt}


\section{Introduction and statement of problem }
We consider the one dimensional Helmholtz equation in a one layered medium:
\begin{equation}
\label{PDE}
u(x,\omega)'' +k^2 u(x,\omega)=f, \quad x\in(-1,1),
\end{equation}
where the wave field $u$ is required to satisfy the outgoing wave conditions:
\begin{equation}
\label{outgoing cond}
   u'(-1,\omega)+iku(-1,\omega)=0, \qquad   u'(1,\omega)-iku(1,\omega)=0
\end{equation}

Given $f \in L^2 (-1,1)$, it is well-known that the problem \eqref{PDE}-\eqref{outgoing cond} has a unique solution:   
\begin{equation}
\label{u}
    u(x,\omega)=\int_{-1}^1 G(x-y) f (y)dy,
\end{equation}
where $G(x)$ is the Green function given as follows
\begin{equation}
\label{GreenF}
    G(x)=\frac{ie^{ik|x|}}{2k}.
\end{equation}
 This work concerns the inverse source problem when the source function $f$ is a complex function with a compact support contained in $(-1,1)$. here our goal is to recover the source function $f$ using the boundary data $u(1,\omega )$ and $u(-1, \Omega)$ with $\omega \in (0,k) $ where $K>1$ is a positive constant. \\

Inverse source problem arises in many area of science. It has numerous applications in  acoustical and biomedical/medical imaging, antenna synthesis, geophysics, and material science (\cite{ABF, B}). It has been known that the data of the inverse source problems for Helmholtz equations with single frequency can not guarantee the uniqueness (\cite{I}, Ch.4). On the other hand, various studies, for instance in \cite{ BLT}, showed that the uniqueness can be regained by taking multifrequency boundary measurement in a non-empty frequency interval $(0,K)$ noticing the analyticity of wave-field on the frequency \cite{I, J}. On the other hand, various studies, for instance in \cite{EV}, showed that the uniqueness can be regained by taking multi-frequency boundary measurement in a non-empty frequency interval (0,K) noticing the analyticity of wave-field on the frequency. Because of the wide applications, these problems have being attracted considerable attention. These kinds of problems have been extensively investigated by many researchers (see, for example,  \cite{AI, BLT1,BLRX,  CIL, E,EI, EG, IK,IL,IL1, J2,   YL,LY}). We also note that these type of problem and technique can apply to systems. For example, in \cite{EI2}, inverse source problems was considered for classical elasticity system. 
In this paper, we intended to prove 

\begin{theorem}
\label{maintheorem}
There exists a generic constant $C$ depending on the domain $\Omega $ such that

\begin{equation}
\label{Istability}
\parallel f\parallel_{(0)}^{2}(-1,1) \leq   C\Big(\epsilon ^2+\frac{M^{2}}{K^{\frac{2}{3}}E^{\frac{1}{4}}+1} \Big),
\end{equation}
for all $u\in H^2 (\Omega)$ solving \eqref{PDE} with $K>1$. Here

\begin{equation*}
\label{epsilon}
\epsilon ^2  = \int _{0} ^{K} \omega ^2\big(  | u(1,\omega) |^{2} + | u(-1,\omega) |^{2} \big ) d\omega,
\end{equation*}
$E=-ln\epsilon $ and $M = \max \big \{ \parallel f \parallel_{(1)}^2(-1,1) , 1 \big\}$ where $\parallel . \parallel _{(l)}(\Omega)$ is the standard Sobolev norm in $H^l(\Omega)$.
\end{theorem}
\begin{center}
\begin{equation*}
f_1 = \begin{cases}
                f & \mbox{if } x > 0, \\
0 & \mbox{if } x < 0,
       \end{cases} \quad
f_{2} = \begin{cases}
            0  & \text{if } x > 0, \\
            f & \text{if } x < 0.
       \end{cases}
\end{equation*}

 \end{center}
 {\bf Remark 1.1: } The estimate in \eqref{Istability} consists of two parts: the data discrepancy and the high frequency part. The first part is of the LIpschitz type. The second part is of logarithmic type. The second part decrease as $K$ increases which makes the problem more stable. The estimate \eqref{Istability} also implies the uniqueness of the inverse source problem.
\section{Proof of Theorem 1.1}
\subsection{Increasing Stability of Continuation to higher frequencies}
Let

\begin{equation*}
    I(k)=I_1(k)+I_2(k)
\end{equation*}
where
\begin{equation}
\label{I1-I2}
I_1(k)= \int _{0} ^{k} \omega ^2 | u(-1,\omega) |^{2} d\omega, \qquad  I_2 (k)= \int _{0}^{k}  \omega ^2 | u(1,\omega) |^{2} d\omega,
\end{equation}
using \eqref{u} and a simple calculation shows that
\begin{equation}
\label{omegau1u-1}
    \omega u(1,\omega)=\int_{0}^{1}\frac{i}{2}e^{i\omega(1- y)}  f_1(y)dy,  \qquad   \omega u(-1,\omega)=\int_{-1}^{0}\frac{i}{2}e^{i\omega (-1- y)}  f_2(y)dy,
\end{equation}
where $y \in (-1,1)$. Functions $I_1$ and $I_2$ are both analytic with respect to the wave number $k \in \mathbb{C} $ and play important roles in relating the inverse source problems of the Helmholtz equation and the Cauchy problems for the wave equations. \\
\begin{lemma}
Let $supp f \in (-1,1)$ and $f \in H^1 (-1,1) $. Then
\begin{equation}
\label{boundI1}
|I_1 (k)|\leq C\Big( |k| \parallel f \parallel_{(0)}^2 (-1,1)  \Big)e^{2|k_2|},
\end{equation}
\begin{equation}
\label{boundI2}
|I_2(k)|\leq C\Big( |k| \parallel f \parallel_{(0)}^2 (-1,1)  \Big)e^{2|k_2|}.
\end{equation}

\end{lemma}
\begin{proof}   

Since we have $k= k_1 +k_2 i$ is complex analytic on the set $\mathbb{S}\setminus [0,k]$, where $\mathbb{S}$ is the sector $S=\{ k \in \mathbb{C}:|$arg$\, k| < \frac{\pi}{4}    \}$ with $k=k_1 +ik_2$. Since the integrands in \eqref{I1-I2} are analytic functions of $k$ in $ \mathbb{S}$, their integrals with respect to $\omega$ can be taken over any path in $\mathbb{S}$ joining points $0$ and $k$ in the complex plane. Using the change of variable $\omega=ks$, $s\in (0,1)$ in the line integral \eqref{u}, the fact that $y\in(-1,1)$.
\begin{equation}
\label{I1}
I_1 (k)= \int _{0} ^{1} ks \big | \int_{0}^{1}\frac{1}{2}e^{i(ks)(1- y)}  f_1(y)dy   \big |^2ds,
\end{equation}
and
\begin{equation}
\label{I2}
    I_2 (k)= \int _{0} ^{1}ks \big | \int_{-1}^{0}\frac{1}{2}e^{i(ks)(-1- y)}  f_2(y)dy   \big |^2ds.
\end{equation}
Noting
\begin{equation*}
    |e^{ iks (-1- y)}|\leq e^{2|k_2|}, \quad |e^{ iks (1- y)}|\leq e^{2|k_2|},
\end{equation*}                                                                                                                                                                                                                                                     
 using the Schwartz inequality and  integrating with respect to $s$, using the bound for $|k|$ in $\mathbb{S}$, we complete the proof of \eqref{boundI1}. Using the same technique, we can prove the \eqref{boundI2}.

\end{proof}
Noticing that functions $I_1(k), I_2(k)$ are analytic functions of $k=k_1+ik_2 \in \mathbb{S}$ and $|k_2|\leq k_1$. The following steps are essential to link the unknown $I_1(k)$ and $I_2(k)$ for $k\in [K,\infty)$ to the known value $\epsilon$ in \eqref{PDE}. Obviously

\begin{equation}
\label{Ie^}
|I_1(k) e^{-2 k}|\leq C\Big( |k_1| \parallel f \parallel_{(0)}^2 (-1,1)\Big)e^{-2 k_1} \leq C M^{2},
\end{equation}
where $M=\max \big\{\parallel f\parallel_{(0)}^2  (-1,1) , 1 \big \}$. With the similar argument bound \eqref{Ie^} is true for $I_2 (k)$. Observing that
\begin{equation*}
|I_1(k)  e^{-2 k}|\leq \epsilon ^2, \quad |I_2(k)  e^{-2 k}|\leq \epsilon ^2  \textit{ on } [0, K].
\end{equation*}
Let $\mu (k) $ be the harmonic measure of the interval $[0,K]$ in $\mathbb{S}\backslash [0,K], $ then as known (for example see \cite{I}, p.67), from two previous inequalities and analyticity of the function $I_{1}(k) e^{-2 k}$ and $ I_{2}(k) e^{-2 k} $ we conclude that
\begin{equation}
\label{I1Epsilon}
|I_1(k) e^{-2 k}|\leq  C\epsilon ^{2 \mu(k)} M^{2},
\end{equation}
when $K<k< +\infty$. Similarly it also yields for
\begin{equation}
|I_2(k) e^{-2k}|\leq  C\epsilon ^{2 \mu(k)} M^{2},
\end{equation}
consequently
\begin{equation}
\label{I2Epsilon}
|I(k) e^{-2 k}|\leq  C\epsilon ^{2 \mu(k)} M^{2}.
\end{equation}

To achieve a lower bound of the harmonic measure $\mu (k)$, we use the following technical lemma. The proof can be found in \cite{CIL}.
\begin{lemma}
Let $\mu (k) $ be the harmonic measure of the interval $[0,K]$ in $\mathbb{S}\backslash [0,K]$, then
\begin{equation}
 \begin{cases}
            \frac{1}{2}\leq \mu (k), & \mbox{if } \quad  0<k<2^{\frac{1}{4}}K, \\
\frac{1}{\pi} \Big( \big ( \frac{k}{K} \big)^{4} -1 \Big)^{\frac{-1}{2}} \leq \mu(k), & \mbox{if } \quad
2^{\frac{1}{4}}K <k.
       \end{cases}
\end{equation}
\end{lemma}
\vspace{0.5cm}

\begin{lemma}
Let source function $f\in L^2 (-1,1)$ with $supp f \subset (-1,1)$, then
\begin{equation*}
\parallel f\parallel^{2} _{(0)} (-1,1) \leq C \int_{0}^{\infty} \omega ^2 \big ( |u(-1,\omega)|^2 + |u(1,\omega)|^2 \big)d\omega.
\end{equation*}
\end{lemma}
\begin{proof}
Using the result of \cite{YL} by applying the Green function \eqref{GreenF} and letting $k_1=k_2=k$.
\end{proof}

\begin{lemma}
Let source function $f\in L^2 (-1,1)$, then
\begin{equation*}
    \omega^2 |u(-1,\omega)|^2  \leq C \Big | \int _{-1}^{0}e^{2\omega y } f_2 (y)dy\Big |^ 2
\end{equation*}
\begin{equation*}
     \omega^2  |u(1,\omega)|^2 \leq C \Big | \int _{0}^{1}e^{2\omega y } f_1 (y)dy\Big |^ 2
\end{equation*}
\end{lemma}
\begin{proof}
It follows from \eqref{omegau1u-1} and $y\in (-1,1)$.
\end{proof}

\subsection{Increasing stability for inverse source problem}
To continue the estimate for reminders in \eqref{I1} and \eqref{I2} for $(k, \infty)$, we need the following lemma.
\begin{lemma}
\label{BoundK-Infty}
Let $u$ be a solution to the forward problem \eqref{PDE} with $f_1 \in H^1(\Omega )$  with $supp f  \subset (-1,1)$, then
\begin{equation}
\label{lemma4.1}
\int_{k}^{\infty} \omega ^2 | u(-1,\omega) |^2 d\omega+ \int_{k}^{\infty}   \omega ^2|u(1,\omega) |^{2} d\omega \leq C k ^{-1}\Big(\parallel f\parallel ^2  _{(1)} (-1,1) \Big)
 \end{equation}
\end{lemma}

\begin{proof}
Using \eqref{omegau1u-1}, we obtain
\begin{equation}
\label{Line1lemma4.1}
   \int _{k}^{\infty} \omega ^2 | u(-1,\omega) |^{2}d\omega + \int _{k}^{\infty} \omega ^2 | u(1,\omega) |^{2}d\omega
\end{equation}

\begin{equation}
\label{Line2lemma4.1}
    \leq C \big ( \int _{k}^{\infty}\Big | \int _{0}^{1} e^{i \omega y}  f_{1} (y)dy  \Big|^2d\omega  + \int _{k}^{\infty}\Big | \int _{-1}^{0} e^{i \omega y}  f_{2} (y)dy  \Big|^2d\omega \Big ).
\end{equation}
Using integration by parts and the fact that $ supp f_1 \subset (0,1)  $ and $supp f_2 \subset (0,1) $, we have
\begin{equation*}
\int _{0}^{1} e^{ -i \omega y}  f _1(y)dy = \frac{1}{ i \omega} \int _{0}^{1}e^{ -i\omega y} ( \partial _ yf_{1}(y) )dy, 
\end{equation*}
and
\begin{equation*}
\int _{-1}^{0} e^{-i \omega y}  f_{2} (y)dy = \frac{1}{i \omega} \int _{-1}^{0}e^{ -i \omega y}(\partial _ yf_{2} (y)  )dy,
\end{equation*}
consequently for the first and second terms in \eqref{Line2lemma4.1} we obtain
\begin{equation*}
   \Big | \int _{0}^{1} e^{i \omega y}  f_{1} (y)dy  \Big|^2 \leq  \frac{C}{\omega^2}\parallel f_{1}\parallel^2 _{(1)} (0,1) \leq  \frac{C}{\omega^2} \parallel f_{1}\parallel^2 _{(1)} (-1,1)
\end{equation*}
\begin{equation*}
    \leq  \frac{C}{\omega^2} \parallel f\parallel^2 _{(1)} (-1,1),
\end{equation*}
utilizing the same argument for the second term in \eqref{Line2lemma4.1} and integrating with respect to $\omega$ the proof is complete.
\end{proof}
Now, we are ready to proof Theorem \ref{maintheorem}.\\
\begin{proof}
We can assume that $\epsilon <1$ and $3\pi E^{-\frac{1}{4}} <1$, otherwise the bound \eqref{maintheorem} is obvious. Let
\begin{center}
\begin{equation}
\label{k}
k= \begin{cases}
          K^{\frac{2}{3}}E^{\frac{1}{4}} \quad \text{if} \quad 2^{\frac{1}{4}}K^{\frac{1}{3}}< E ^{\frac{1}{4}} \\
K \hspace{1.19 cm} \text{if}\quad E ^{\frac{1}{4}} \leq 2^{\frac{1}{4}}K^{\frac{1}{3}},
       \end{cases}
\end{equation}
 \end{center}
if $ E ^{\frac{1}{4}} \leq 2^{\frac{1}{4}}K^{\frac{1}{3}}$, then $k=K$, using the \eqref{I1Epsilon} and \eqref{I2Epsilon},  we can conclude
\begin{equation}
\label{I11}
|I (k)| \leq 2\epsilon ^2.
 \end{equation}

If $2^{\frac{1}{4}}K^{\frac{1}{3}}< E ^{\frac{1}{4}}$,
we can assume that $ E^{-\frac{1}{4}} <\frac{1}{4 \pi}$, otherwise $C<E$ and hence $K<C$ and the bound \eqref{Istability} is straightforward. From \eqref{k},  Lemma 2.2, \eqref{I1Epsilon} and the equality $\epsilon= \frac{1}{e^ E}$ we obtain
\begin{equation*}
\label{I1epsilon}
|I(k) |\leq  CM^{2} e^{4k}  e^{\frac{-2E}{\pi}\big( (\frac{k}{K})^4  -1 \big)^{\frac{-1}{2}}}
\end{equation*}

\begin{equation*}
\leq CM^{2}  e^{- \frac{2}{\pi}K^{\frac{2}{3}}E^{\frac{1}{2}}(1- \frac{5\pi}{2}   E^{\frac{-1}{4}} )},
\end{equation*}
using the trivial inequality $e^{-t} \leq \frac{6}{t^3}$ for $t>0$ and our assumption at the beginning of the proof, we obtain
\begin{equation}
\label{I12}
|I(k)|\leq  CM ^{2}  \frac{1}{K^2 E ^{\frac{3}{2}}\Big (1-\frac{5 \pi}{2} E^{-\frac{1}{4}}   \Big)^3}. 
\end{equation}
Due to the \eqref{I1}, \eqref{I11}, \eqref{I12}, and Lemma 2.5. we can conclude
\begin{equation}
\label{lastbound1}
\int ^{+\infty}_{0} \omega ^2 |u(-1,\omega)|^2 d\omega + \int ^{+\infty}_{0} \omega ^2 |u(1,\omega)|^2 d\omega
\end{equation}
\begin{equation*}
    \leq I(k)+  \int_{k}^{\infty} \omega ^2 |u(-1,\omega)|^2 d\omega + \int_{k}^{\infty} \omega ^2 |u(1,\omega)|^2 d\omega
\end{equation*}
\begin{equation*}
\leq 2\epsilon ^2 +   \frac{ C M^2 } {K^2 E ^{\frac{3}{2}}} + \frac{ \parallel f \parallel_{(2)}^2 (-1,1)} {K^{\frac{2}{3}}E^{\frac{1}{4}}+1} \Big).
\end{equation*}
Using the inequalities in \eqref{lastbound1} and Lemma 2.3., we finally obtain
\begin{equation*}
\parallel f \parallel _{(0)} ^2 (\Omega)\leq C \Big( \epsilon ^2 +   \frac{   M^{2} }{K^2 E ^{\frac{3}{2}}} + \frac{ \parallel f \parallel_{(1)}^2(-1,1) }{K^{\frac{2}{3}}E^{\frac{1}{4}}+1} \Big)
\end{equation*}
Due to the fact that $ K^{\frac{2}{3}} E ^{\frac{1}{4}}<K^2 E ^{\frac{3}{2}}$ for $1<K, 1<E$, the proof is complete.

\end{proof}

\section{Concluding Remarks} In this paper, we studied the inverse source problem with many frequencies in a one dimensional domain. The result showed that if $K$ grows the estimate improves. It also showed that if we have date exists for all wave number $k\in (0,\infty) $, the estimate will be a Lipschitz estimate. \\

{\bf Acknowledgment:}
We wish to thank Dr. M.N. Entekhabi for her support and encouragement.

$^a$ Northern Boarder University, Saudi Arabia;\\
E-mail address: Shahah.Almutairi@nbu.edu.sa\\

\vspace{0.2cm}

\hspace{-0.6cm}$^a$ Department of Mathematics, Florida A \&\ M University, Tallahassee, FL 32307, USA\\
E-mail address:  ajith.gunaratne@famu.edu

\end{document}